\documentclass[a4paper,fleqn]{cas-sc}

\usepackage[authoryear,longnamesfirst]{natbib}

\def\tsc#1{\csdef{#1}{\textsc{\lowercase{#1}}\xspace}}
\tsc{WGM}
\tsc{QE}
\newtheorem{theorem}{Theorem}
\newtheorem{example}[theorem]{Example}
\newtheorem{lemma}[theorem]{Lemma}
\newdefinition{definition}{Definition}
\newdefinition{rmk}{Remark}
\newproof{pf}{Proof}
\usepackage[percent]{overpic}

\DeclareMathOperator{\ch}{ch}
\DeclareMathOperator{\aff}{aff}

\DeclareMathOperator{\inte}{int}

\DeclareMathOperator{\Rext}{Rext}

\DeclareMathOperator{\Rch}{Rch}
\DeclareMathOperator{\Extrem}{ext}

\newcommand{\Lat}{\mathscr{K}}
\newcommand{\RLat}{\mathscr{R}}
\newcommand{\Real}{\mathbb{R}}

\newcommand{\id}{\mathrm{id}}
\begin{document}
\let\WriteBookmarks\relax
\def\floatpagepagefraction{1}
\def\textpagefraction{.001}

\shorttitle{Finite versus uncountable convex lattices from point configurations}  

\shortauthors{Carles Cardó}  

\title [mode = title]{Finite versus uncountable convex lattices from point configurations}




\author[1]{Carles Cardó}[orcid=0000-0003-3836-3635]

\cormark[1]


\ead{ccardo@uic.es}



\affiliation[1]{organization={Universitat Internacional de Catalunya},
            addressline={c/Josep Trueta s/n},  
            city={Sant Cugat del Vallès},
%          citysep={}, % Uncomment if no comma needed between city and postcode
            postcode={08195}, 
            state={Catalunya},
            country={Spain}}

\cortext[1]{Corresponding author}



\begin{abstract}
We study the smallest convex lattice generated by a finite set of points. To analyze this structure, we introduce the notion of a point configuration, defined via the relative lattice. Under a suitable completeness condition, this lattice becomes a combinatorial counterpart of the convex lattice and is therefore easier to handle. We investigate the enumeration of these structures and prove that, while the number of relative lattices is always finite, the number of convex lattices is uncountable for $n \geq 6$.
\end{abstract}


\begin{highlights}
\item Relative convex lattices provide a combinatorial model for point configurations.
\item Under a natural completeness condition, convex and relative lattices coincide.
\item There are only finitely many relative lattices on $n$ points.
\item In contrast, convex lattices are uncountable for $n \geq 6$.
\end{highlights}

% Keywords
\begin{keywords}
Convex geometry \sep Lattice theory \sep Point configurations \sep Convex lattices \sep Relative lattices
\end{keywords}

\maketitle

%%%%%%%%%%%%%
\section {Introduction}\label{SectionIntro}
%%%%%%%%%%%%%
The \emph{convex hull} of a set $X \subseteq \mathbb{R}^d$, denoted by $\ch(X)$, is the smallest convex set containing $X$. 
The symbol $\ch$ is a \emph{closure operator}, and therefore the family of convex sets in $\mathbb{R}^d$, denoted by $\Lat$, is a complete lattice with operations $X\vee Y= \ch (X \cup Y)$ and $X \wedge Y=X \cap Y$; see \cite{brondsted2012introduction}.

A simple way to obtain a sublattice of $\Lat$ is to consider an initial finite set of points $X$, draw segments, triangles, and more generally, convex hulls, and then intersect them to obtain new points. With this process, we construct the smallest convex sublattice containing the points of $X$ as singletons, which we will denote by $\Lat(X)$. We will call these lattices \emph{point-generated lattices}. Since the intersection and the convex hull of a pair of bounded polytopes are bounded polytopes, by structural induction, the elements of $\Lat(X)$ are bounded polytopes. Notice, however, that $\Lat(X)$ does not always have a greatest element. Consider, for example, the infinite set $X=\{1/n \mid n \geq1 , n\in \mathbb{N}\}$. However, when $X$ is finite, then $\ch(X)$ is the greatest element. 

For most sets of points $X$, the generated lattice $\Lat(X)$ is infinite. For example, take the vertices of a regular pentagon. By drawing all possible segments, we obtain a smaller pentagon inside the original pentagon, and by repeating the operation, we obtain a third pentagon, and so on. We may ask for which initial sets of points in the plane the process stops. These configurations can be grouped into the following families:
\begin{enumerate}[(i)]
\item  \emph{Linear configurations}, $L_n$, consisting of a set of $n$ collinear points. 
\item \emph{Triangular configurations}, $T_n$, consisting of a set of $n$ collinear points and an extra non-collinear point.  
\item \emph{Diamond configurations}, $D_{p,q}$, given by two sets of collinear points, $A$ and $B$, such that $A\cap B=\{c\}$ and $|B|=3$. The point $c$ splits $A\setminus\{c\}$ into two subsets with $p$ and $q$ points.   
 \item \emph{Subdiamond configurations}, $I_{p,q}=D_{p,q}\setminus \{c\}$, where $c$ is the point of $D_{p,q}$ defined above.
 \item The \emph{sporadic configuration}, $S_6$, consisting of six points $\{a,b,c,a',b',c'\}$ with the only collinear subsets $\{a,a',c'\}$, $\{b,b',a'\}$ and $\{c,c',b'\}$.\end{enumerate}
 
%%%%%%%%%%%%%%%%%%%%%%
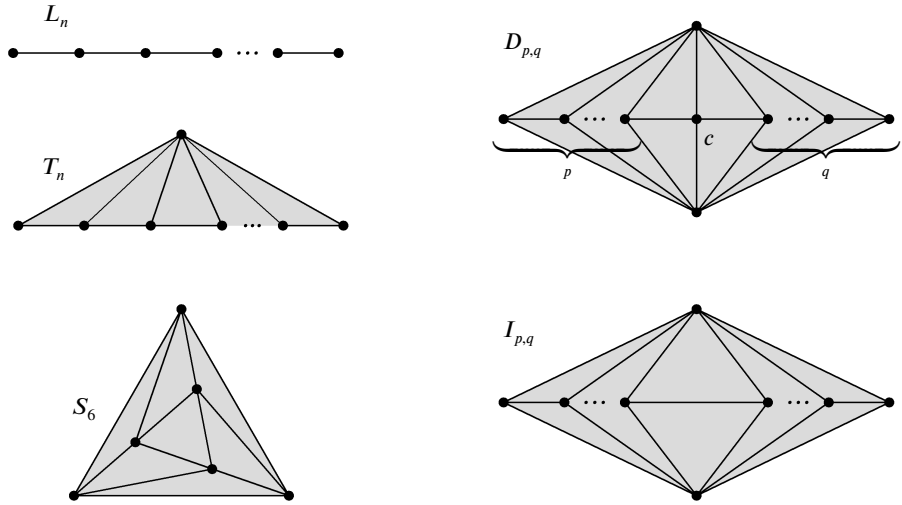
\begin{figure}
\centering
\begin{overpic}[scale=0.150]{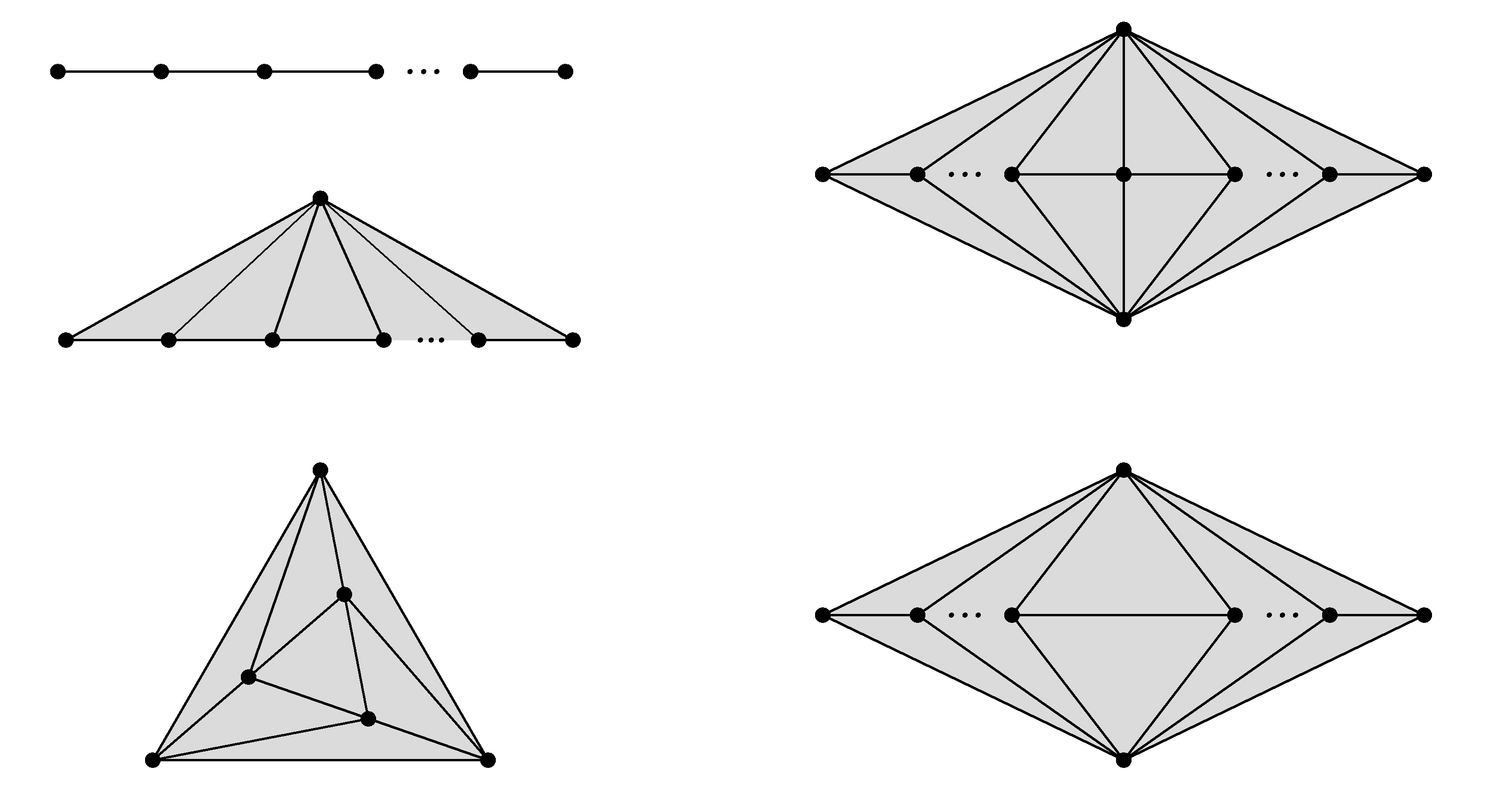}
 \put (7,53) {$L_n$} 

 \put (7,37) {$T_n$} 
 
 \put (10,12) {$S_6$}  
 
 \put (55,50) {$D_{p,q}$} 
 \put (76,40) {$c$}
 \put (54,41) {$\underbrace{\qquad \qquad \qquad }_{p}$}
 \put (81,41) {$\underbrace{\qquad\qquad \qquad }_{q}$}
    
 \put (55,20) {$I_{p,q}$} 

\end{overpic}
\caption{The unique planar configurations $X$ for which $\Lat(X)$ is finite.}\label{FigDefClasses}
\end{figure}
%%%%%%%%%%%%%%%%%%%%%% 
See Figure~\ref{FigDefClasses}. We can even summarize it further: a configuration gives rise to a finite convex lattice if and only if it is either a subconfiguration of the diamond configuration, or is the sporadic configuration $S_6$.  See \cite{Eppstein2026} and \cite{Grune2008} for proofs of this result.

There are several problems and questions related to this one. For example, one can consider the problem of determining which configurations in $\mathbb{R}^3$ generate a finite convex lattice. To the best of our knowledge, this question remains open. Another example is the following: in which regions of the plane the set of points generated by the lattice is dense? See \cite{Klein2006} for proofs, and see \cite{Caceres2007, Ismalescu2004, Cooper2010}, among others, for other variants of the problem.

Some configurations, such as those mentioned above, are so simple that they can be described in terms of the collinearity of their points. However, to study more general configurations in a systematic way, we need to define configurations in abstract terms, not just as a set of points. This set needs to be given a structure that takes into account certain invariant properties. This depends, of course, on the problem we are investigating. In the case of this article, we are not interested in metric properties, but rather in incidence, collinearity, coplanarity, and related properties. In this sense, we will see that the \emph{relative lattice} $\RLat(X)$ of a set $X$ of points allows us to define the notion of ``configuration'' in an intrinsic way, appropriate for problems such as the finiteness of $\Lat(X)$. We explore this concept in the next Section~\ref{SectionRela}.

When $X$ is finite, the relative lattice $\RLat(X)$ is also finite. Then it becomes a more manageable structure, but it does not reflect the whole structure of $\Lat(X)$, and some information is lost. However, we will see in Section~\ref{SectionComplete} that, under a suitable \emph{completeness} condition, the lattices $\RLat(X)$ and $\Lat(X)$ are isomorphic, in which case $\RLat(X)$ becomes the combinatorial version of $\Lat(X)$. This constitutes our first structural result.

The second goal of the article, presented in Section~\ref{SectionNonIso}, concerns the enumeration of these two structures up to isomorphism. In the case of relative lattices, the number grows exponentially but is always finite for a fixed number of points. In the second case of convex lattices, we show that, surprisingly, for $n\geq 6$ the number of lattices is infinite and uncountable. This is the second main result of this paper.

Let us fix the minimal notation for the rest of the article. When convenient, we will abbreviate the singletons of a lattice as $a=\{a\}$ and the segments as $ab=\{a\}\vee \{b\}=a\vee b$. Given three non-collinear points $a,b,c$ in the plane, $\triangle abc$ denotes the triangle with vertices $a,b,c$, that is, $\triangle abc=a\vee b \vee c$.

To denote that three distinct points $a,b,c$ are collinear and that $b$ lies between $a$ and $c$, we will write $a{-}b{-}c$.

Given $X \subseteq \mathbb{R}^d$, we denote by $\inte (X)$ the interior of a set (more precisely, the interior relative to its affine span). By $\Extrem(X)$ we denote the set of extreme points of $X$.

The remaining terms will be recalled or introduced throughout the article. However, for elementary concepts in algebra, we will refer to \cite{Burris2012}; for questions in order theory to \cite{davey2002introduction}; and for questions in convex geometry, we will follow \cite{brondsted2012introduction}.

%%%%%%%%%%%%%%%%%%%%%%%%%
%%%%%%%%%%%%%%%%%%%%%%%%%
\section{Relative lattices and configurations} \label{SectionRela}
%%%%%%%%%%%%%%%%%%%%%%%%%
%%%%%%%%%%%%%%%%%%%%%%%%%

To provide a suitable notion of ``configuration'' of points, we introduce \emph{relative convex lattices}. The term is taken from a survey on convex sets by \cite{bergman2005lattices} in a broader context, whereas our interest is more combinatorial.

\begin{definition}
Given a subset $A \subseteq X \subseteq \Real^d$, the \emph{relative convex hull of $A$ with respect to $X$} is the set
$$\Rch_X(A)=\ch(A) \cap X.$$
When the ambient set $X$ is clear from the context, we simply write $\Rch(A)$. A subset $A \subseteq X$ is said to be \emph{relatively convex} if $\Rch(A)=A$. Since $\Rch(\cdot)$ is a closure operator, we can define the \emph{relative lattice of $X$} as the set of relatively convex sets,
$$ \RLat(X)=\{ \Rch_X(Y) \mid Y\subseteq X\},$$
with operations $A\vee_R B=\Rch_X(A\cup B)$ and $A\wedge_R B=A \cap B$.

We say that two sets $X$ and $Y$ are \emph{equivalent}, and write $X\equiv Y$, if there exists a lattice isomorphism between $\RLat(X)$ and $\RLat(Y)$. We will use the term \emph{configuration} for such an equivalence class.
 By abuse of notation, we will identify a configuration with any of its representatives.
\end{definition}

\begin{example}
The linear configuration $L_n$, as defined at the beginning, Section~\ref{SectionIntro}, is a configuration in the abstract sense. Any set of $n$ collinear points yields the same relative lattice up to isomorphism. The same occurs for the other configurations $T_n, D_{p,q}, I_{p,q}, S_6$: their definitions are given by means of collinearity and incidence. In fact, these are pairwise non-equivalent  (except for a few degenerate cases, such as $L_2$ and $T_1$, or $T_3$ and $D_{0,1}$).
\end{example}

\begin{example}
Let $T=\{a,b,c,d\}$ be the vertices of a triangle $\triangle abc$ together with a point $d$ in its interior, and let $S$ be the set of vertices of a square. Then, $T$ and $S$ are not equivalent, since $d \in \Rch(\{a,b,c\})$, while no triangle determined by points of $S$ contains another point of $S$.

We can consider other relations between sets beyond equivalence. For example, let three collinear points $L_3=\{x,y,z\}$ with $x{-}y{-}z$. There exists a surjective lattice morphism $f:\RLat(T) \longrightarrow \RLat(L_3)$ given by
\[
f(\{a\})=\{x\},\quad f(\{b\})=f(\{c\})=\{z\},\quad f(\{d\})=\{y\}.
\]
\end{example}

Contrary to geometric intuition, the dimension of a configuration (defined as the dimension of its affine hull) is not invariant under equivalence. For instance, a tetrahedron is equivalent to a convex quadrilateral. More generally, it is straightforward to see that any set of $n$ points in convex position in $\Real^d$ is equivalent to a regular $n$-gon in some plane of $\Real^d$.

Fortunately, there are many effective invariants. Of course, equivalent configurations have the same cardinality. Two useful invariants are the following. First, define the set of \emph{relative extreme points} of $X$ as
\[ \Rext(X)=\{ x \in X \mid X\setminus \{x\} \text{ is relatively convex} \}. \]
If $X\equiv Y$, then $\Rext(X)\equiv \Rext(Y)$ as configurations.
The second invariant is, in fact, a family of invariants. Let $Z$ be a configuration. We define
\[ \#_Z(X)=\text{the number of subconfigurations of $X$ equivalent to $Z$}. \]

Invariants help to separate configurations, but do not indicate how to find all of them. We observe, however, that if $f:\RLat(X) \longrightarrow \RLat(Y)$ is an isomorphism, and $a\in X$, then $X\setminus\{a\} \equiv Y\setminus \{ f(\{a\})\}$. Thus, configurations of size $n+1$ can be obtained from those of size $n$ by adding one point. Depending on the position of the added point, different configurations arise. However, this process does not ensure that we capture all configurations of a fixed order. Despite this, it can be shown that the number of configurations with a fixed number of points is finite; see Section~\ref{SectionNonIso}.

What we can do is list the planar configurations of small orders that we know to be pairwise non-equivalent. See Figure~\ref{FigPlanarConf} for the designations.

There is a unique configuration of size one and of size two, denoted $L_1$ and $L_2$. From $L_2$ we get two configurations of size three. When we add a point on the line of $L_2$, we get  $L_3$ if the point is on the line, and $T_2$ if it is not. These are not equivalent since $|\Rext(L_3)|=2$, while  $|\Rext(T_2)|=3$. 

For configurations of size four, there are four configurations: $L_4$, $T_3, I_{0,2}$, and $I_{1,1}$. The following table shows that for each pair of configurations, there is at least one invariant that separates them.    

\begin{center}
\begin{tabular}{ c | c c c c }
 & $L_4$ & $T_3$ & $I_{0,2}$ & $I_{1,1}$ \\ 
 \hline
 $\#_{L_1}(X)$ &4 &4& 4& 4\\  
 $\#_{L_2}(X)$ &6 &6& 6& 6\\
 $\#_{L_3}(X)$ &4 &1& 0& 0\\
 $\#_{T_2}(X)$ &0 &3& 4& 4\\
 $|\Rext(X)|$ &2 &3& 3& 4\\
\end{tabular}
\end{center}

Using the same invariants, we have been able to find twelve pairwise non-equivalent configurations of order five. We conjecture that this list is complete, although we do not have a formal proof.

Let us note that the invariants $\#_Z( \cdot)$ satisfy the equalities:
$$\#_{L_3}(X)+ \#_{T_2}(X) ={ |X| \choose 3},$$
for $|X|\geq 3$, and if $|X|\geq 4$,
$$\#_{L_4}(X)+ \#_{T_3}(X) +\#_{I_{0,2}}(X)+ \#_{I_{1,1}}(X) ={ |X| \choose 4}.$$
More generally, when $|X|\geq k$,
$$ \sum_{Z} \#_Z(X)= { |X| \choose k},$$
where the sum runs over all the configurations $Z$ of size $k$.  The proof is straightforward, since any $k$-element subset of $X$ must be equivalent to some $k$-point configuration.

%%%%%%%%%%%%%%%%%%%%%%
\begin{figure}
\centering
\begin{overpic}[scale=0.18]{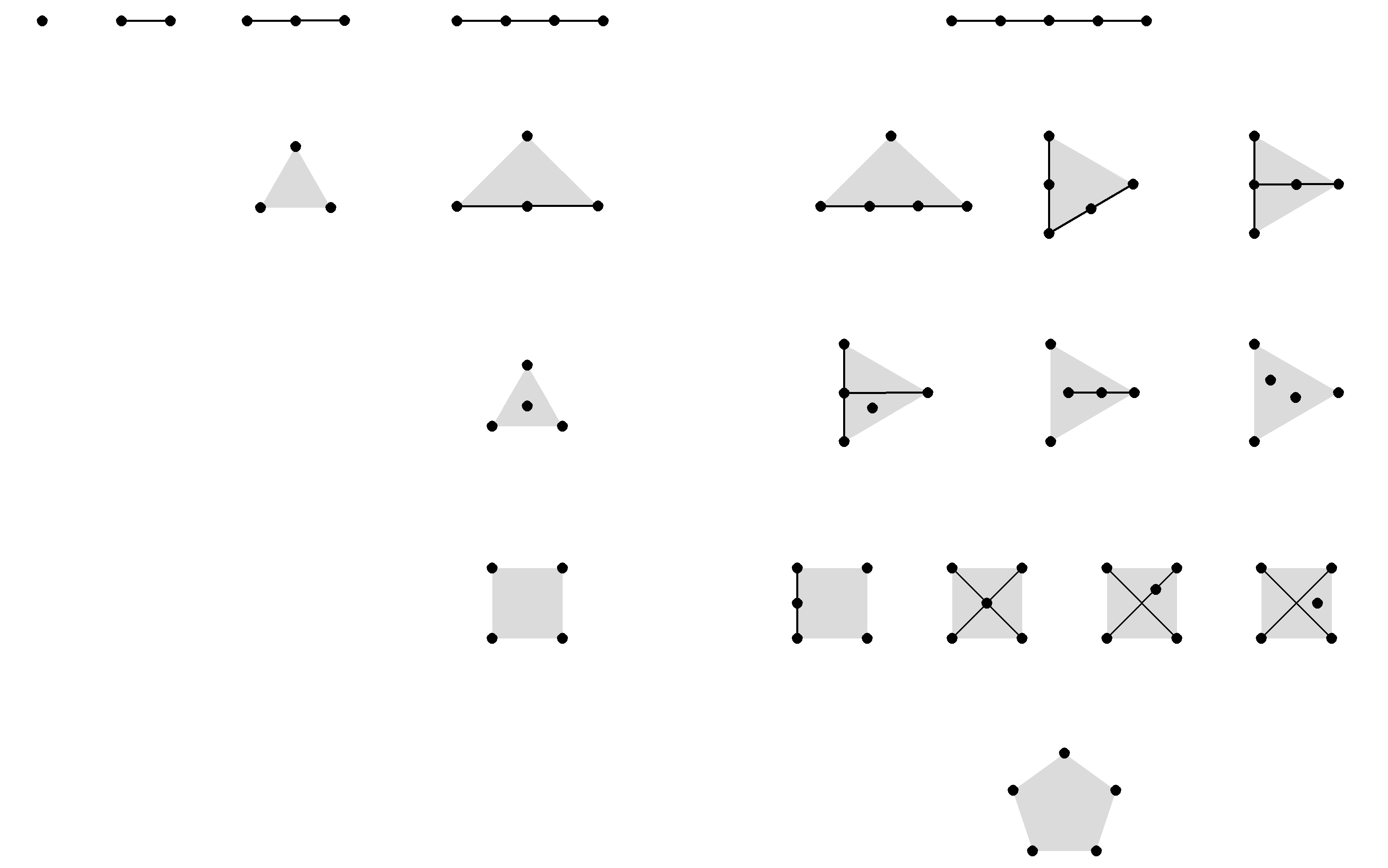}
 \put (2,58) {$L_1$} 
 \put (9,58) {$L_2$} 
 \put (20,58) {$L_3$} 
 \put (37,58) {$L_4$} 
 \put (75,58) {$L_5$} 

  \put (17,53) {$T_2$} 
  \put (33,53) {$T_3$} 
  \put (66,53) {$T_4$} 
  \put (78,53) {$V_5$} 
  \put (92,53) {$D_{0,2}$}

  \put (35.5,38.5) {$I_{0,2}$} 
  \put (63,38) {$R$} 
  \put (78,38) {$I_{0,3}$} 
  \put (93,38) {$G$} 

  \put (36,24) {$I_{1,1}$}

  \put (58.5,24) {$R'$} 
  \put (69,24) {$D_{1,1}$} 
  \put (80,24) {$I_{1,2}$} 
  \put (92,24) {$G'$} 
     
  \put (70.5,8) {$P_5$} 
  
\end{overpic}
\caption{Planar configurations with 1, 2, 3, 4, and 5 points and their designations. See  Section~\ref{SectionIntro} for the labels $L_{n}$, $T_n$, $D_{p,q}$, and $I_{p,q}$. The remaining designations are arbitrary. Shadows indicate the convex hulls. Lines indicate either collinear points or forbidden positions for additional points. }\label{FigPlanarConf}
\end{figure}
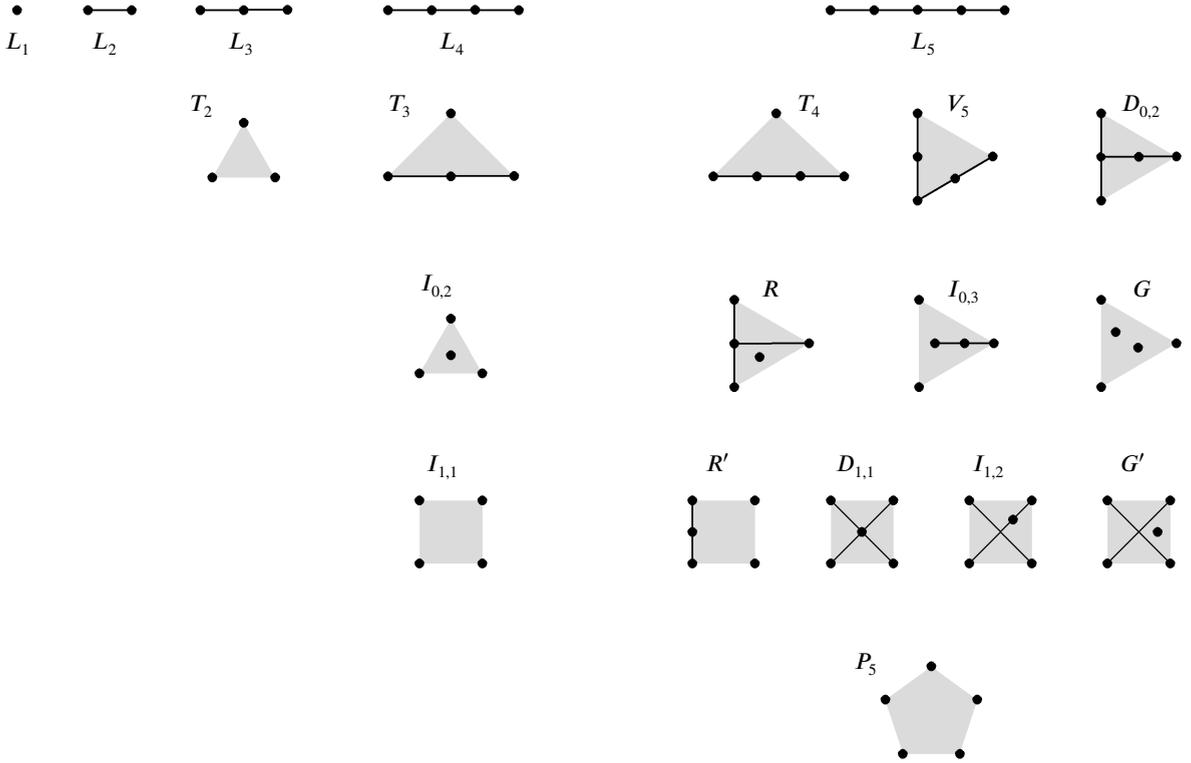
%%%%%%%%%%%%%%%%%%%%%%

%%%%%%%%%%%%%%%%%%%%%%%%%%
\section{Complete configurations} \label{SectionComplete}
%%%%%%%%%%%%%%%%%%%%%%%%%%

\begin{definition} Given a configuration $X$, its \emph{completion} is the set of singleton elements of $\Lat(X)$, i.e., all points that can be generated by lattice operations. It is denoted by  $\overline{X}=\{ x \in \mathbb{R}^d \mid \{x\} \in \Lat(X) \}$. The set $X$ is said to be \emph{complete} if $\overline{X}=X$. 
\end{definition}
Notice that $\overline{X}$ is always complete and that $\Lat(X)=\Lat(\overline{X})$. 
Completeness means that no new points arise from lattice operations.

\begin{example} If we add the center point to a convex quadrilateral $I_{1,1}$, we complete the configuration $\overline{I}_{1,1}\equiv D_{1,1}$. We have $\Lat(I_{1,1})= \Lat(\overline{I}_{1,1})$. 
\end{example}

We will prove that if $X$ is complete, $\RLat(X)$ is the combinatorial version of $\Lat(X)$, or more precisely, they are isomorphic. We first need a lemma.
\begin{lemma} \label{LemmaExtreme} Let $X$ be complete. For any $C\in \Lat(X)$, $\Extrem(C) \subseteq X$. 
\end{lemma}
\begin{pf} Let the singleton $\{a\} \in \Lat(X)$. Since $X$ is complete and $\Extrem(\{a\})=\{a\}$, $a\in X$. Suppose that $A,B \in \Lat(X)$ are such that $\Extrem(A), \Extrem(B) \subseteq X$. If we prove that $\Extrem(A \wedge B)\subseteq X$, and $\Extrem(A \vee B)\subseteq X$, then, by structural induction, for any $C\in \Lat(X)$, $\Extrem(C) \subseteq X$. 

For the join operation, we have directly that 
$$\Extrem(A \vee B) \subseteq \Extrem (A) \cup \Extrem (B) \subseteq X \cup X=X.$$
For the meet operation, recall that $A$ and $B$ are bounded polytopes, and the intersection of bounded polytopes is a bounded polytope. There are several algorithms in computational geometry to obtain the intersection of two polytopes. However, we only need to know that the extreme points of $A \wedge B$ are obtained from intersections of faces of $A$ and $B$;  see \cite{Preparata2012}. Since the faces of $A$ and $B$ are convex hulls of their extreme points, 
$$\Extrem(A\wedge B) \subseteq \Lat( \Extrem(A)\cup \Extrem(B) )\subseteq \Lat( X \cup X)=\Lat( X).$$
By completeness of $X$, the singleton elements of $\Lat(X)$ are precisely the elements of $X$. Then, $\Extrem(A\wedge B) \subseteq X.$
\end{pf}

\begin{theorem} \label{PropositionIso} If $X$ is complete, $\Lat(X) \cong \RLat(X)$.
\end{theorem}
\begin{pf} Let $\Phi:  \Lat(X) \longrightarrow \RLat(X)$, defined by $\Phi(A)=A\cap X$, and let  $\ch: \RLat(X) \longrightarrow \Lat(X)$ be the convex hull operator. We prove that $\Phi$ and $\ch$ are inverses of each other. 
 
If $A \in \RLat(X)$, then $A$ is relatively convex, $\Rch(A)=A$, and then  $\Phi(\ch(A))=\ch(A)\cap X=\Rch(A)=A$. Hence, $\Phi \circ \ch= \id_{\RLat(X)}$.

Let $B \in \Lat(X)$. We want to prove that $B=\ch(B\cap X)$. Consider first the inclusion $(\subseteq)$.  By Lemma~\ref{LemmaExtreme}, $\Extrem (B) \subseteq X$. Clearly $\Extrem (B) \subseteq B$, and then $\Extrem (B) \subseteq B\cap X$. Since $\ch$ is monotone, $\ch(\Extrem B) \subseteq \ch(B \cap X)$, that is, $B\subseteq \ch(B\cap X)$, where we have used, by Minkowski's theorem, that a bounded polytope is the convex hull of its extreme points, see \cite{brondsted2012introduction}.
 
 Now we prove the inclusion $(\supseteq)$. We have that $\ch(B \cap X) \subseteq \ch(B)$. Since $B$ is a convex set, $\ch(B)=B$, and then $\ch(B \cap X) \subseteq B$. Hence, $\ch ( \Phi (B))=\ch ( B \cap X)=B$. Thus, $\ch \circ \, \Phi=\id_{\Lat(X)}$. 

Next, we check that $\Phi$ is a morphism of lattices. For the meet operation, it is trivial: 
\begin{align*}
\Phi(A\wedge B)&=\Phi(A \cap B) =A\cap B \cap X \\
&=A \cap X \cap B \cap X=\Phi(A)\cap \Phi(B)\\
&=\Phi(A)\wedge_R \Phi(B). \end{align*}
For the join operation, $\Phi(A \vee B)=\ch(A\cup B)\cap X$. Notice that we cannot write $\ch(A\cup B)\cap X= \Rch(A\cup B)$, because $A$ and $B$ are not subsets of $X$, but of $\mathbb{R}^d$. However, as shown above $\ch ( B \cap X) = B$. Therefore,
\begin{align*}
\Phi(A \vee B) &= \ch(A\cup B)\cap X\\
 &= \ch((A\cup B)\cap X) \cap X\\
&=\ch((A\cap X)\cup (B\cap X)) \cap X\\
&=\Rch((A\cap X)\cup (B\cap X))\\
&=\Rch(A\cap X)\vee_R \Rch(B\cap X)\\
&=(\ch(A\cap X) \cap X)\vee_R (\ch(B\cap X)\cap X)\\
&=(A \cap X)\vee_R (B\cap X)=\Phi(A)\vee_R \Phi(B).
\end{align*}
\end{pf}

%%%%%%%%%%%%%%%%%%%%%%%%%%%%%
\section{Number of non-isomorphic relative and convex lattices} \label{SectionNonIso}
%%%%%%%%%%%%%%%%%%%%%%%%%%%%%

We now consider the lattice enumeration problem in the planar case. Let $R_n$ denote the number of non-isomorphic relative lattices, that is, non-equivalent planar configurations, and $K_n$ denote the number of non-isomorphic convex lattices.
In Section~\ref{SectionRela} we saw that for every $n\leq 5$, the values $R_n$ satisfy the following lower bounds
$$1, \;1, \:2, \:4, \:12.$$
These numbers are conjectured to be exact. For $n\geq 6$, we cannot provide a close estimate. However, we can easily show that, on the one hand, the growth is at least exponential, and on the other hand, that the number $R_n$ is finite.

\begin{theorem} The number of non-equivalent planar configurations with $n$ points
$$1.58^{n-6}\leq R_n \leq 2^{n 2^n}.$$
\end{theorem}
\begin{pf} Since $\Rch_X$ is a closure operator on $2^X$, the number of non-isomorphic lattices is bounded above by the number of closure operators, which in turn is bounded above by the number of all maps $2^X \to 2^X$, namely $(2^n)^{2^n}$.

For the lower bound, consider a sequence of nested triangles $(\triangle_k)_{k\geq 1}$,
$$\triangle _k \subseteq \inte (\triangle _{k+1}).$$ 
With $n$ points, we can form a sequence of $\lfloor n/3 \rfloor$ triangles. Each pair of consecutive triangles can realize one of four distinct collinearity patterns between their vertices: 
\begin{enumerate}
 \item No three points are collinear in $\triangle_k \cup \triangle_{k+1}$. 
 \item There are exactly three collinear points. 
 \item There are two sets of collinear points. 
 \item There are three sets of collinear points, that is, $\triangle_k \cup \triangle_{k+1}$ is the sporadic configuration $S_6$. 
 \end{enumerate}
For each inclusion of consecutive triangles, we can decide how to align the vertices. If two such constructions of triangles yield equivalent configurations, then the sets of collinear points for each pair of consecutive triangles must be the same. Therefore, with $n$ points, we can form the following number of non-equivalent  sequences of triangles:
$$4^{\lfloor \frac{n}{3} \rfloor-1}\geq 4^{\frac{n}{3}-2}=\left(4^\frac{1}{3}\right)^{n-6}.$$
\end{pf}

Let us now consider the sequence $K_n$. It is easy to verify that $K_n=1, 1, 2, 4$, for $n=1,2,3,4$. For $n=5$, we can also see that $K_5\geq 12$. On the one hand, thanks to Theorem~\ref{PropositionIso}, if the configuration $X$ is complete, $\Lat(X)\cong \RLat(X)$. In particular, the complete configurations $L_5, T_4, D_{0,2}, D_{1,1}$ are all distinct. As for the configurations $I_{0,3}$ and $I_{1,2}$, they are not complete, but when a suitable point is added, they become complete, giving rise to the non-equivalent configurations $D_{0,3}$ and $D_{1,2}$. Their convex lattices $\Lat(D_{0,3})$ and $\Lat(D_{1,2})$ are not isomorphic, and therefore $\Lat(I_{0,3})$ and $\Lat(I_{1,2})$ are not isomorphic to any of the others.

Next, consider the configurations $R,R',G,G',V_5$ and $P_5$, which are neither complete nor generate a finite convex lattice. To distinguish them, we can use invariants. On the one hand, the number $|\Extrem(X)|$ is also an invariant of convex lattices. On the other hand, one can also verify that the number $|X\cap \partial \ch (X)|$ (where $\partial$ denotes the boundary) is also an invariant. Thus, the twelve known configurations yield at least twelve non-isomorphic lattices.

This suggests that for $n\leq 5$, we may have $R_n=K_n$.
However, for $n\geq 6$, the situation changes completely. While $R_n$ is finite, we will show that $K_n=2^{\aleph_0}$.

In order to compute $K_n$, for $n\geq 6$, we need some preliminary lemmas regarding the structure of a very specific configuration.
 Let $V_5=\{o, a, b, a', b'\}$ be the configuration with $o{-}a{-}a'$, $o{-}b{-}b'$, and no other collinear triples. For convenience, we will take $o=(0,0)$, $a'=2a$, and $b'=2b$.
This configuration gives rise to an infinite convex lattice, as we will see later. We need an elementary property of universal algebra. If $F:L \longrightarrow L$ is a lattice automorphism and $L$ has a unique minimal generating set, say $G$, then $G$ is invariant, $F(G)=G$.

\begin{lemma} \label{LemmaPropV5} The configuration $V_5$ satisfies the following properties.
\begin{enumerate}[(i)]
\item There are only two lattice automorphisms of $\Lat(V_5)$: the identity $\id$ and the involution $S$ that satisfies $S(\{o\})=\{o\}$, $S(\{a\})=\{b\}$ and $S(\{a'\})=\{b'\}$, which we will call the symmetry of $V_5$.
\item If $d,d' \in \inte ( \triangle oab)$, then any isomorphism $\Lat(V_5\cup\{d\}) \longrightarrow \Lat(V_5\cup\{d'\})$ must preserve $V_5$ and send the singleton $d$ to $d'$.
\end{enumerate}
\end{lemma}
\begin{pf} (i) Every lattice automorphism transforms generating sets into generating sets. One can verify that $V_5$ is the unique minimal generating set of $\Lat(V_5)$. This means that any isomorphism will transform $V_5$ into $V_5$. Consider the map $f:V_5 \longrightarrow V_5$ given by $f(x)=y$ if and only if $F(\{x\})=\{y\}$.

Let us see that $f$ can only take two forms: either $f$ is the identity or $f$ exchanges $a$ for $b$ and $a'$ for $b'$. Given that $o{-}a{-}a'$, in order for $f$ to extend to a morphism, it must satisfy that $f(o){-}f(a){-}f(a')$. On the other hand $o$, $a'$ and $b'$ are extreme points, and therefore $f(o)$, $f(a')$ and $f(b')$ must also be extreme. Finally, $a'$ and $b'$ have no intermediate point in $\Lat(V_5)$, and therefore the same must hold for the points $f(a')$ and $f(b')$. These three observations combined suffice to show that $f$ can only be the identity or the symmetry.

It remains only to see that indeed these two possible $f$ extend to a lattice morphism. As for the identity, this is obvious. As for the second case, there exists an affine map $S:\Real^2\longrightarrow \Real^2$ consisting of a reflection with axis $0c'$, where $c'=ab'\wedge ba'$, and which induces a lattice automorphism.

Let us consider point (ii). Let $F:\Lat(V_5\cup\{d\}) \longrightarrow \Lat(V_5\cup\{d'\})$ be an isomorphism. 

All new points generated from the configuration $V_5$ by the lattice are inside the triangle $\triangle abc'$, where $c'=ab'\wedge ba'$. In \cite{Grune2008}, it is shown that, in addition, the set of points is dense in that triangle, but there is no new point outside it. 
Thus, $\inte(\triangle oab)\cap \overline{V_5}=\emptyset$, and therefore, the point $d$ cannot be generated from $V_5$. Hence $V_5\cup\{d\}$ is the unique minimal generating set of $\Lat(V_5 \cup\{d\})$ and the same holds for $\Lat(V_5 \cup \{d'\})$. Since $F$ must preserve the minimal generating sets that are unique, $F$ sends the points of $V_5\cup\{d\}$ to the points of $V_5\cup\{d'\}$. Furthermore, $F$ must respect the boundary; the only possibility remaining is that $F(\{d\})=\{d'\}$.
\end{pf}

To be able to differentiate non-isomorphic convex lattices, we will use the following construction. Consider the points
$$c'=ab'\wedge ba' \quad \mbox{ and } \quad c=oc' \wedge ab.$$
The points $\{o,c,c',b,b'\}$ again form a configuration that is equivalent to $V_5$. And likewise the points $\{0,a,a',c,c'\}$, see Figure~\ref{FigContinuum}. Thus, for each of these configurations, we could repeat the construction and find new configurations, all equivalent to $V_5$.

We are interested in the triangles that are formed by means of the vertex $o$ and the vertices that appear in the segment $ab$ as a result of iterating the previous construction. But to refer to them, we need a systematic encoding. We will use only the elementary notation on strings and formal languages, see for example \cite{Hopcroft1979}.
Let $\Sigma=\{0,1\}$ be an alphabet with two letters. $\Sigma^*$ denotes the set of strings (or words) over the alphabet $\Sigma$. $\Sigma^n$ denotes the set of words with length $n$ and $\Sigma^\omega$ denotes the set of words of infinite length. We denote by $\alpha^\omega$ the infinite word $\alpha\alpha\alpha\cdots$. Given a word $\alpha=\alpha_1\cdots\alpha_n \in \Sigma^*$, $\alpha_i \in \Sigma$, we define the \emph{complementary word} as $\widetilde{\alpha}=\widetilde{\alpha_1} \cdots \widetilde{\alpha_n}$ where $\widetilde{0}=1$ and $\widetilde{1}=0$. The complementary word of an infinite-length word is defined by extension.

As noted above the point $c$ subdivides the triangle $\triangle oab$ into two triangles: an upper triangle $\triangle_0$ and a lower triangle $\triangle_1$. We can repeat these subdivisions indefinitely:
\begin{align*}
\triangle oab &=\triangle_0 \cup \triangle_1 \\
&=\triangle_{00} \cup \triangle_{01} \cup \triangle_{10} \cup \triangle_{11} \\
&=\triangle_{000} \cup \triangle_{001} \cup \triangle_{010} \cup \triangle_{011} \cup \triangle_{100} \cup \triangle_{101} \cup \triangle_{110} \cup \triangle_{111}\\
&=\cdots
\end{align*}

%%%%%%%%%%%%%%%%%%%%%%
\begin{figure}
\centering
\begin{overpic}[scale=0.15]{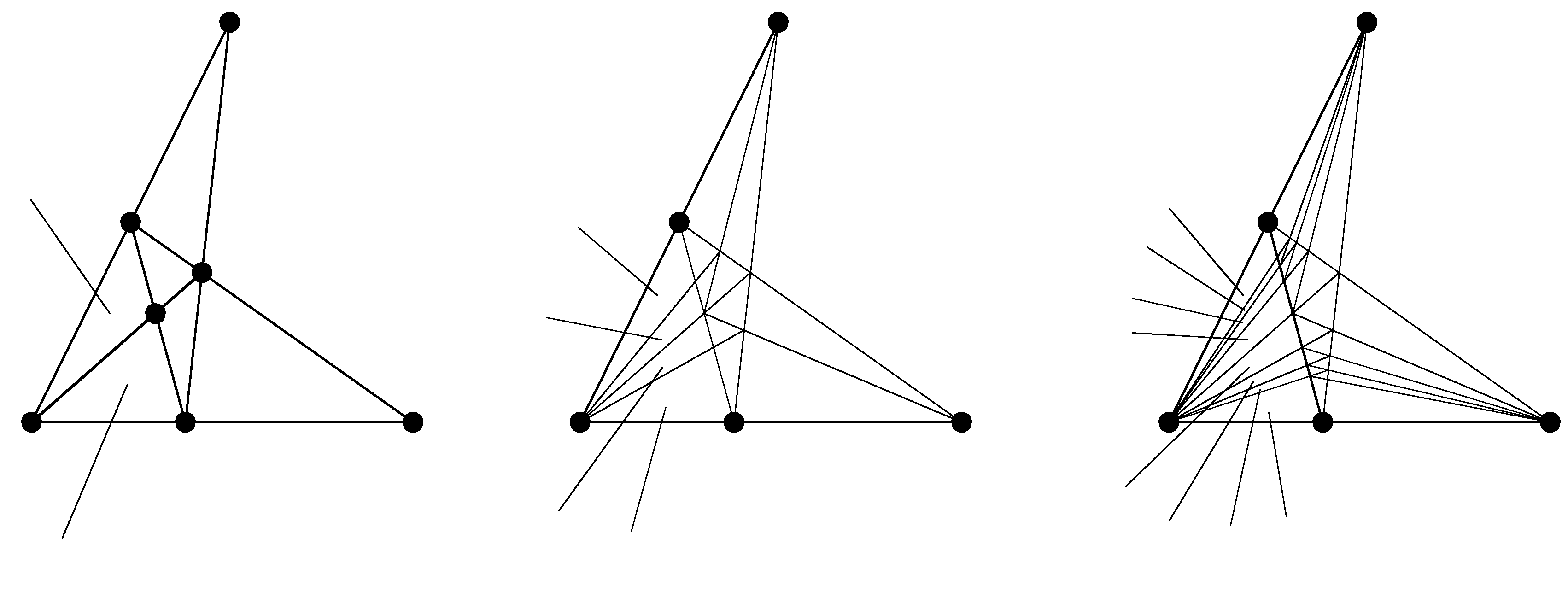}

 \put (0,27) {$\triangle_0$} 
 \put (1,1.5) {$\triangle_1$} 
 
 \put (0, 9) {$o$}
 \put (10, 9) {$a$}
 \put (25.5,9) {$a'$}
 \put (11.5,37) {$b'$}
 \put (6.5,26) {$b$}
 \put (9.8,20.7) {$c$}
 \put (13.5,23) {$c'$}

 \put (32,25.5) {$\triangle_{00}$} 
 \put (29,18) {$\triangle_{01}$} 
 \put (30.5,4.5) {$\triangle_{10}$} 
 \put (37,2) {$\triangle_{11}$} 

 \put (73,26.5) {$\triangle_{000}$} 
 \put (66.5,24) {$\triangle_{001}$} 
 \put (65,20) {$\triangle_{010}$} 
 \put (65.5,16) {$\triangle_{011}$} 
 
 \put (64,7.5) {$\triangle_{100}$} 
 \put (68,3.75) {$\triangle_{101}$} 
 \put (75.5,1.5) {$\triangle_{110}$} 
 \put (82.5,3) {$\triangle_{111}$} 
 
\end{overpic}
\caption{Triangle subdivision of the proof of Theorem~\ref{TheContinuum}.} \label{FigContinuum}
\end{figure}
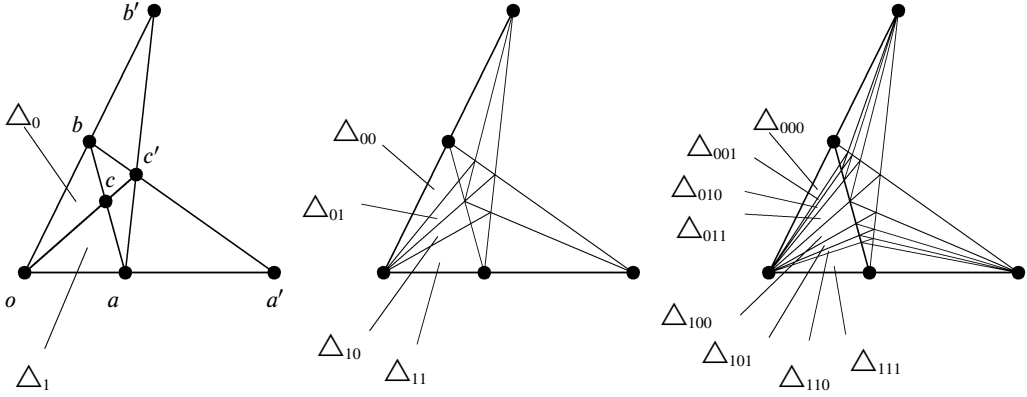
%%%%%%%%%%%%%%%%%%%%%%

\begin{lemma} \label{LemmaCantor} Let $\alpha=\alpha_1\alpha_2\cdots \in \Sigma^\omega$ be an infinite word. The limit
$$\liminf_{n\to \infty} \triangle_{\alpha_1\cdots \alpha_n}=\bigcap_{n\geq 1} \triangle_{\alpha_1\cdots \alpha_n}=\triangle_\alpha$$
exists, and it is non-empty. In addition, $\triangle_\alpha$ contains at least a segment $oe$ for some point $e \in ab$. 
\end{lemma}
\begin{pf} Given a triangle $\triangle_{\alpha_1\cdots \alpha_n}$, denote by $E_{\alpha_1\cdots \alpha_n}$ the opposite side to the vertex $o$.  First, we notice that $\triangle_{\alpha_1\cdots \alpha_n} \supseteq \triangle_{\alpha_1\cdots \alpha_n\alpha_{n+1}}$. Then, $E_{\alpha_1\cdots \alpha_n} \supseteq E_{\alpha_1\cdots \alpha_n\alpha_{n+1}}$. The sets $E_{\alpha_1\cdots \alpha_n}$ are compact, and by Cantor's intersection theorem, the limit
$$\liminf_{n\to \infty} E_{\alpha_1\cdots \alpha_n}=\bigcap_{n\geq 1} E_{\alpha_1\cdots \alpha_n}=E_\alpha$$
 exists, and it is non-empty, see for example \cite{Munkres2012}. Thus, there is at least some point $e\in E_\alpha$. Therefore, the limit $\triangle_\alpha$ exists and 
 contains a segment $oe$ for some point $e \in ab$. 
\end{pf}

We still need another technical result. 
In the proof of our final theorem, a crucial step consists of establishing a bijection between triangles and words. This situation is analogous to the fact that some numbers in binary code have two equivalent expressions, for example, $0.1000111\ldots$ and $0.1001000\ldots$. A solution is to fix a normal form. In the case of real numbers, the first encoding is discarded. However, in our  proof, we remove this ambiguity by excluding eventually constant sequences.

We say that two finite length words, $\alpha, \beta \in \Sigma^*$ are \emph{contiguous} if there is a $\gamma \in \Sigma^*$ such that $\alpha=\gamma10^m$ and $\beta=\gamma01^m$ (or conversely), for some integer $m\geq 0$. The definition also extends to infinite-length words as $\alpha=\gamma 10^\omega$, $\beta=\gamma 01^\omega$. We say that two triangles $\triangle_\alpha, \triangle_\beta$ are contiguous if their words are. Intuitively, contiguous words correspond to triangles sharing a boundary segment.

\begin{lemma} Let $\alpha, \beta \in\Sigma^\omega$. If $\triangle_\alpha \cap \triangle_\beta \neq \emptyset$, then either the triangles are the same or contiguous. 
\end{lemma}
\begin{pf} First, we consider the case for finite-length words. Let $\alpha, \beta \in \Sigma^*$, different words with $|\alpha|=|\beta|=n$. We proceed by induction on the length. For $n=1$, the statement is trivial. Suppose that it is true for any word with length less than $n+1$, and let $\alpha, \beta \in \Sigma^*$, with $\alpha\not=\beta$, $|\alpha|=|\beta|=n+1$, and  
 $\alpha=\alpha_1\cdots \alpha_{n+1}$, $\beta=\beta_1\cdots \beta_{n+1}$. 
 
 We consider two cases. If $\alpha_1=\beta_1=0$, then we can assume that we are in the initial subdivision process of the upper triangle, that is, consider the $V_5$ configuration. Thus, if $\triangle_{\alpha_1\cdots \alpha_{n+1}}\cap \triangle_{\beta_1\cdots \beta_{n+1}}\not=\emptyset$, then $\triangle_{\alpha_2\cdots \alpha_{n+1}}\cap \triangle_{\beta_2\cdots \beta_{n+1}}\not=\emptyset$. By the induction hypothesis, $\alpha_2\cdots \alpha_{n+1}$, $\beta_2\cdots \beta_{n+1}$ are contiguous, and therefore $\alpha_1\cdots \alpha_{n+1}$, $\beta_1\cdots \beta_{n+1}$ are contiguous. The argument is symmetric if we suppose $\alpha_1=\beta_1=1$.  
 
Suppose now that $\alpha_1\not= \beta_1$. We have two subcases: $\alpha_1=0$ and $\beta_1=1$, and $\alpha_1=1$,  and $\beta_1=0$. We only examine the first case, since the second one is symmetric. 
The statement holds, since the interiors of $\triangle_0$ and $\triangle_1$ are disjoint and only meet along a segment. If the intersection of the triangles is not empty, the words necessarily continue as $\alpha=01^{n}$ and $\beta=10^{n}$. 

The next step is to extend this result to infinite-length words. Clearly $\triangle_{\alpha_1\cdots \alpha_n} \cap \triangle_{\beta_1\cdots \beta_n} \supseteq \triangle_{\alpha_1\cdots \alpha_{n+1}} \cap \triangle_{\beta_1\cdots \beta_{n+1}}$, and again by Cantor's theorem, the limit set
$$\bigcap_{n\geq 1} \left( \triangle_{\alpha_1\cdots \alpha_n} \cap \triangle_{\beta_1\cdots \beta_n} \right)=\triangle_\alpha \cap \triangle_\beta \not=\emptyset,$$
provided that $\triangle_{\alpha_1\cdots \alpha_n} \cap \triangle_{\beta_1\cdots \beta_n}\not=\emptyset$ for each $n\geq1$.  If $\triangle_{\alpha_1\cdots \alpha_n} \cap \triangle_{\beta_1\cdots \beta_n} \not=\emptyset$, then $\alpha_1\cdots \alpha_n$ and $\beta_1\cdots \beta_n$ are contiguous, and then, in the limit, $\alpha=\alpha_1\alpha_2\cdots$, $\beta_1\beta_2\cdots$ are contiguous. Therefore, if $\triangle_\alpha \cap \triangle_\beta\not=\emptyset$, then $\alpha, \beta$ are contiguous. 
\end{pf}

\begin{theorem} \label{TheContinuum} For $n\geq 6$, $K_n=2^{\aleph_0}$.
\end{theorem}
\begin{pf}  
Consider the following set of words,
$$\Gamma=\Sigma^\omega \setminus \left(\Sigma^*01^\omega \cup \Sigma^*10^\omega \right).$$
The set $\Gamma$ contains no contiguous words, and therefore, $\triangle_\alpha\cap \triangle_\beta=\emptyset$, for any $\alpha, \beta \in \Gamma$ with $\alpha\not=\beta$. In particular $\triangle_\alpha=\triangle_\beta$ iff $\alpha=\beta$.
Thus, the map $\alpha \mapsto \triangle_\alpha$ defines a bijection between $\Gamma$ and the set $\{\triangle_\alpha \mid \alpha \in \Gamma\}$.
 Since $\Sigma^*01^\omega \cup \Sigma^*10^\omega$ is countable and $\Sigma^\omega$ has cardinality $2^{\aleph_0}$, it follows that $|\Gamma|=2^{\aleph_0}$.

Let $\alpha, \beta \in \Gamma$ and consider the points $z_\alpha \in \inte (\triangle oab) \cap \triangle_\alpha$ and $z_\beta \in \inte (\triangle oab) \cap \triangle_\beta$. These points exist by Lemma~\ref{LemmaCantor}. By Lemma~\ref{LemmaPropV5}(ii), any isomorphism $\Lat( V_5 \cup \{z_\alpha\}) \longrightarrow \Lat(V_5 \cup \{z_\beta\})$ must send $z_\alpha$ to $z_\beta$, and hence must map the corresponding triangle $\triangle_\alpha$ to $\triangle_\beta$. However, if we restrict the isomorphism to $\Lat(V_5)$ we obtain an automorphism. 
By Lemma~\ref{LemmaPropV5}(i), $\Lat(V_5)$ has only two automorphisms: $\id$ and $S$. It is easy to see that for any triangle $\triangle_\alpha$, with $\alpha\in \Sigma^*$, $\id(\triangle_\alpha) =\triangle_\alpha$ and $S(\triangle_\alpha) =\triangle_{\widetilde{\alpha}}$. Moreover, this also holds for limit triangles, $\triangle_\alpha$ where $\alpha$ is a word of infinite length.

Finally, let us pick the following words $0\alpha$ and $0\beta$, with $\alpha\not=\beta$, $\alpha, \beta \in \Gamma$, that is, $z_{0\alpha}, z_{0\beta}$ are both in the triangle $\triangle_0$. Then, lattices $\Lat(V_5 \cup \{z_{0\alpha}\})$ and $\Lat(V_5 \cup \{z_{0\beta}\})$ cannot be isomorphic. If $F$ is an isomorphism, it cannot be the identity because the points are different, and it cannot be $S$ because 
$0\alpha\not=\widetilde{0\beta}=1\widetilde{\beta}$. Since the cardinality of $\{ 0\alpha \mid \alpha \in \Gamma\}$ is $2^{\aleph_0}$, we have a continuum of essentially different lattices. 

For $n \geq 6$, these configurations can be extended to $n$-point configurations while preserving non-isomorphism.
\end{pf}

A natural question is how these results extend to other closure systems. The most immediate analogue is the affine hull operator $\aff(X)$ of a set of points. This operator also induces a lattice $\Lat_{\mathrm{aff}}(X)$, namely the smallest lattice of affine subspaces containing $X$. However, several phenomena observed in the convex setting no longer hold. For instance, while the points of $\Lat(V_5)$ are dense only in the triangle $\triangle abc'$, in the affine case the points generated by the lattice operations are dense in the entire plane; see \cite{Ismalescu2004}. Moreover, the automorphism group of $\Lat(V_5)$ is $\mathbb{Z}_2$, which makes the structure easy to control, whereas the automorphism group of $\Lat_{\mathrm{aff}}(V_5)$ appears to be much more complex. It remains an open problem to determine whether an analogue of Theorem~\ref{TheContinuum} holds in the affine setting, or whether fundamentally different phenomena arise.

%%%%%%%%%%%%%%%%%%%%%%%%%%%%%%%%%%%%%%%%%%%%%%


%\section{}\label{}


\bibliographystyle{cas-model2-names}

\bibliography{ConvexBib}



\end{document}